\documentclass{amsart}
\usepackage{amsmath,amssymb,amsthm}
\usepackage{graphicx}
\usepackage{fullpage}
\usepackage{mathtools}
\usepackage{tikz}
\usetikzlibrary{matrix,arrows}
\usepackage{color}
\usepackage{enumitem}
\usepackage{pdfsync}     
\usepackage{hyperref}
\usepackage{placeins}

\newtheoremstyle{myremark} 
    {7pt}                    
    {7pt}                    
    {}  	                 
    {}                           
    {\bf}       	         
    {.}                          
    {.5em}                       
    {}  

\theoremstyle{plain}
\newtheorem{lemma}{Lemma}[section]
\newtheorem{theorem}[lemma]{Theorem}
\newtheorem*{theorem-main}{Theorem~\ref{thm:main}}
\newtheorem{definition}[lemma]{Definition}

\newtheorem{proposition}[lemma]{Proposition}

\theoremstyle{definition}

\theoremstyle{myremark}

\newcommand{\N}{\mathbb{N}}

\newcommand{\R}{\mathbb{R}}

\newcommand{\sphere}{\mathbb{S}}

\newcommand{\diam}{\mathrm{diam}}
\newcommand{\cost}{\mathrm{cost}}

\newcommand{\conv}{\mathrm{conv}}

\newcommand{\cP}{\mathcal{P}}
\newcommand{\ball}[2]{\mathrm{B}(#1,#2)}
\newcommand{\cball}[2]{\overline{\mathrm{B}}(#1,#2)}

\newcommand{\inj}{\hookrightarrow}     
\newcommand{\suchthat}{~\middle|~}

\newcommand{\vr}[2]{\mathrm{VR}(#1;#2)}
\newcommand{\vrleq}[2]{\mathrm{VR}_\leq(#1;#2)}
\newcommand{\vrless}[2]{\mathrm{VR}_<(#1;#2)}
\newcommand{\vrm}[2]{\mathrm{VR}^m(#1;#2)}
\newcommand{\vrmleq}[2]{\mathrm{VR}^m_\leq(#1;#2)}
\newcommand{\vrmless}[2]{\mathrm{VR}^m_<(#1;#2)}
\newcommand{\cech}[2]{\mathrm{\check{C}}(#1;#2)}
\newcommand{\cecha}[3]{\mathrm{\check{C}}(#1,#2;#3)}
\newcommand{\cechaleq}[3]{\mathrm{\check{C}_{\le}}(#1,#2;#3)}
\newcommand{\cechaless}[3]{\mathrm{\check{C}_{<}}(#1,#2;#3)}
\newcommand{\cechm}[2]{\mathrm{\check{C}^m}(#1;#2)}

\newcommand{\cecham}[3]{\mathrm{\check{C}^m}(#1,#2;#3)}
\newcommand{\cechamleq}[3]{\mathrm{\check{C}_\leq^m}(#1,#2;#3)}

\newcommand{\id}{\mathrm{id}}

\newcommand{\Tub}{\mathrm{Tub}}

\newcommand{\comp}{\circ}
\newcommand{\supp}{\mathrm{supp}}

\begin{document}

\title{Metric Thickenings of Euclidean Submanifolds}
\author{Henry Adams}
\address{Department of Mathematics, Colorado State University, Fort Collins, CO 80523, United States}
\email{adams@math.colostate.edu}
\author{Joshua Mirth}
\address{Department of Mathematics, Colorado State University, Fort Collins, CO 80523, United States}
\email{mirth@math.colostate.edu}
\date{\today}

\begin{abstract}
Given a sample $Y$ from an unknown manifold $X$ embedded in Euclidean space, it is possible to recover the homology groups of $X$ by building a Vietoris--Rips or \v{C}ech simplicial complex on top of the vertex set $Y$.
However, these simplicial complexes need not inherit the metric structure of the manifold, in particular when $Y$ is infinite. Indeed, a simplicial complex is not even metrizable if it is not locally finite.
We instead consider metric thickenings, called the \emph{Vietoris--Rips} and \emph{\v{C}ech thickenings}, which are equipped with the 1-Wasserstein metric in place of the simplicial complex topology.
We show that for Euclidean subsets $X$ with positive reach, the thickenings satisfy metric analogues of Hausmann's theorem and the nerve lemma (the metric Vietoris--Rips and \v{C}ech thickenings of $X$ are homotopy equivalent to $X$ for scale parameters less than the reach).
To our knowledge this is the first version of Hausmann's theorem for Euclidean submanifolds (as opposed to Riemannian manifolds), and our result also extends to non-manifold shapes (as not all sets of positive reach are manifolds).
In contrast to Hausmann's original proof, our homotopy equivalence is a deformation retraction, is realized by canonical maps in both directions, and furthermore can be proven to be a homotopy equivalence via simple linear homotopies from the map compositions to the corresponding identity maps.

\end{abstract}
\maketitle

\section{Introduction}

The Vietoris--Rips simplicial complex $\vr{X}{r}$ of a metric space $X$ at scale parameter $r > 0$ has $X$ as its vertex set, and a simplex $\sigma$ for every finite set of points of diameter less than $r$.
Vietoris--Rips complexes are a natural way to enlarge a metric space. 
Indeed, Hausmann proves in~\cite{Hausmann} that given a compact Riemannian manifold $X$ and a sufficiently small scale parameter $r$, the Vietoris--Rips complex $\vr{X}{r}$ is homotopy equivalent to $X$.
In response to a question in Hausmann's paper, Latschev~\cite{Latschev} proves furthermore that if $Y \subseteq X$ is a sufficiently dense sample, then $\vr{Y}{r}$ is also homotopy equivalent to $X$.
Latschev's result is a precursor to many theoretical guarantees~\cite{AFV,AttaliLieutier,ChazalDeSilvaOudot2013,ChazalOudot2008,NiyogiSmaleWeinberger} showing how Vietoris--Rips complexes and related constructions can recover topological information about a shape $X$ from a sufficiently dense sampling $Y$.
In applications of topology to data analysis~\cite{Carlsson2009,EdelsbrunnerHarer} the datasets will typically be finite, but nevertheless infinite Vietoris--Rips constructions are important for applications in part because if a dataset $Y$ converges to an infinite shape $X$ in the Gromov--Hausdorff distance, then the persistent homology of $\vr{Y}{r}$ converges to that of the infinite object $\vr{X}{r}$~\cite{ChazalDeSilvaOudot2013}.

Despite theoretical guarantees such as Hausmann's theorem, the simplicial complex $\vr{X}{r}$ does not retain the metric properties of $X$.
In fact, a simplicial complex is metrizable if and only if it is locally finite, which $\vr{X}{r}$ need not be when $X$ is infinite.
Furthermore, if $X$ is not discrete then the natural inclusion map $X \inj \vr{X}{r}$ is not continuous for any $r>0$. 
The Vietoris--Rips thickening of $X$, denoted $\vrm{X}{r}$ and introduced in~\cite{MetricReconstructionViaOptimalTransport}, addresses each of these issues.
As a set, $\vrm{X}{r}$ is naturally identified with the geometric realization of simplicial complex $\vr{X}{r}$, but it has a completely different topology induced from the 1-Wasserstein metric.
Indeed, the space $\vrm{X}{r}$ is a \emph{metric thickening} of $X$, meaning that it is a metric space extending the metric on $X$.
As a result, the inclusion $X \inj \vrm{X}{r}$ is continuous for all metric spaces $X$ and scale parameters $r$.
In general, the simplicial complex $\vr{X}{r}$ and metric thickening $\vrm{X}{r}$ are neither homeomorphic nor homotopy equivalent, and we argue that the metric thickening is often a more natural object.

In particular, let $X$ be a compact Riemannian manifold. If $X$ is of dimension at least one, then the inclusion $X\inj \vr{X}{r}$ is not continuous. For $r$ sufficiently small, the homotopy equivalence $\vr{X}{r}\xrightarrow{\simeq}X$ in Hausmann's result depends on the choice of a total ordering\footnote{One could use the axiom of choice to pick such a total order, though constructive total orders may also exist.} 
of the points in $X$, meaning it is non-canonical as different choices of orderings produce different maps.
By contrast, the inclusion $X \inj \vrm{X}{r}$ to the metric thickening is continuous, and for $r$ sufficiently small it has as a homotopy inverse the canonical map $\vrm{X}{r} \to X$ defined by Karcher means~\cite{MetricReconstructionViaOptimalTransport,Karcher1977}.

In this work we prove a metric analogue of Hausmann's result for subsets of Euclidean space with positive reach. Our main result is the following:

\begin{theorem-main}
Let $X$ be a subset of Euclidean space $\R^n$, equipped with the Euclidean metric, and suppose the reach $\tau$ of $X$ is positive. Then for all $r < \tau$, the metric Vietoris--Rips complex $\vrm{X}{r}$ is homotopy equivalent to $X$.
\end{theorem-main}

In particular, if $X$ is a submanifold of $\R^n$ with positive reach, then its Vietoris--Rips thickening is homotopy equivalent to $X$ for sufficiently small scale parameters.
To our knowledge, this is the first version of Hausmann's theorem for Euclidean (and hence typically non-Riemannian) manifolds, using either Vietoris--Rips complexes or Vietoris--Rips thickenings.
However, our theorem does not require $X$ to be a manifold, and indeed not every set of positive reach is a manifold~\cite{FedererCurvature,rataj2017structure}.

We prove the main theorem by showing that the linear projection of $\vrm{X}{r}$ into $\R^n$ has image contained in the tubular neighborhood of radius $\tau$ about $X$.
We then map each point in the tubular neighborhood to its unique closest point in $X$.
The composition of these maps produces a homotopy equivalence $\vrm{X}{r}\xrightarrow{\simeq} X$ whose homotopy inverse is the (now continuous) inclusion $X\inj \vrm{X}{r}$.

We provide the following motivation for our work.
Given a sample $Y$ from an unknown shape $X$, there are theoretical guarantees~\cite{AFV,AttaliLieutier,ChazalDeSilvaOudot2013,ChazalOudot2008,Latschev,NiyogiSmaleWeinberger} which recover topological information about $X$ by building a Vietoris--Rips complex or a related construction on top of $Y$. 
However, many of these theorems require choosing a scale parameter (such as $r$ in $\vr{Y}{r}$) which is sufficiently small depending on the curvature of $X$, and hence it may be difficult to choose an appropriate $r$ without prior knowledge of $X$.
Instead, practitioners often compute the persistent homology of $\vr{Y}{r}$ over a range of scale parameters $r$ in order to learn about $X$.
This is well-motivated since as a dataset $Y$ converges to an infinite shape $X$ in the Gromov--Hausdorff distance, stability implies that the persistent homology of $\vr{Y}{r}$ converges to that of $\vr{X}{r}$~\cite{ChazalDeSilvaOudot2013}.
Nevertheless, very little is known about the theory of Vietoris--Rips complexes of infinite shapes (such as manifolds) at larger scale parameters $r$.
This is in part due to the difficulty of working with the simplicial complex $\vr{X}{r}$, which may not be metrizable even though $X$ is a metric space, and for which the natural inclusion map $X\hookrightarrow \vr{X}{r}$ is often not continuous.
We believe that the metric Vietoris--Rips thickening $\vrm{X}{r}$ is often a more natural object to consider when $X$ is infinite.
As evidence of this claim, we provide a metric analogue of Hausmann's theorem ($\vrm{X}{r}\simeq X$ for $r$ sufficiently small) when $X$ is a Euclidean subset of positive reach.
Our results are relevant for data analysis since a finite dataset $Y$ may converge (as more samples are drawn) in the Gromov--Hausdorff distance to an underlying Euclidean set $X$, and since one may want to recover not only the homotopy type or persistent homology of $X$ but also its metric properties.

An important clarifying remark about metrics on Euclidean submanifolds is the following.
Let $X$ be a manifold embedded in Euclidean space $\R^n$ and equipped with the Euclidean submetric (as in Theorem~\ref{thm:main}).
A different metric on $X$ is the Riemannian structure inherited from the usual inner product in $\R^n$---assuming $X$ is connected, the Riemannian distance on $X$ is also a metrization of the original manifold topology~\cite{Lee}. 
The Vietoris--Rips complex, and its homotopy type, may depend upon 
which of these two metrics one chooses to use.
For example, a circle and an ellipse in $\R^2$ with the Riemannian distance function (i.e.\ the arc-length metric) and equal circumferences have identical Vietoris--Rips complexes.
On the other hand, with the Euclidean metrics their Vietoris--Rips complexes are not homotopy equivalent at all scale parameters~\cite{AA-VRS1,AAR}.
The Vietoris--Rips thickening $\vrm{X}{r}$ will similarly depend on the choice of metric on $X$; in this paper we consider the Euclidean metric.

A related construction to the Vietoris--Rips complex is the \v{C}ech complex.
For $X\subseteq\R^n$, the \v{C}ech complex $\cech{X}{r}$ is the nerve simplicial complex of the collection of balls $\ball{x}{r/2}$ with centers $x\in X$. 
The nerve lemma implies that $\cech{X}{r}$ is homotopy equivalent to the union of the balls.\footnote{Here we mean ambient \v{C}ech complexes corresponding to Euclidean balls, though in this paper we also consider intrinsic \v{C}ech complexes corresponding to possibly non-contractible balls in $X$.}
However, the \v{C}ech complex $\cech{X}{r}$ need not inherit any metric properties of $X$, and again is not metrizable if it is not locally finite.
We therefore consider the metric \v{C}ech thickening $\cechm{X}{r}$ from~\cite{MetricReconstructionViaOptimalTransport}, which is a metric space equipped with the 1-Wasserstein distance that furthermore is a metric thickening of $X$.
In Theorem~\ref{thm:cech} we prove that if $X$ is a subset of Euclidean space of positive reach $\tau$, then for all $r<\tau$ the metric \v{C}ech thickening $\cech{X}{2r}$ is homotopy equivalent to $X$.\footnote{This result does not follow from the nerve lemma since the nerve complex $\cech{X}{r}$ and metric thickening $\cechm{X}{r}$ can in general have different homotopy types. The theorem holds for both ambient and intrinsic \v{C}ech thickenings.}
The proof mirrors that of the main theorem, Theorem~\ref{thm:main}.

In Section~\ref{sec:prelims} we review notation and we introduce the Wasserstein metric, Vietoris--Rips and \v{C}ech simplicial complexes, Vietoris--Rips and \v{C}ech metric thickenings, Euclidean subsets of positive reach, and tubular neighborhoods.
Section~\ref{sec:Euclidean-Hausmann} contains our main result, a metric analogue of Hausmann's theorem for Vietoris--Rips thickenings of Euclidean subsets of positive reach, and the lemmas building up to it.
We use similar techniques to prove a version for \v{C}ech thickenings in Section~\ref{sec:Euclidean-nerve}.

\section{Preliminaries}\label{sec:prelims}

We describe background material and notation on metric spaces, Euclidean space, topological spaces~\cite{hatcher2002algebraic}, the Wasserstein metric~\cite{vershik2013long,edwards2011kantorovich}, simplicial complexes, Vietoris--Rips and \v{C}ech simplicial complexes~\cite{EdelsbrunnerHarer}, Vietoris--Rips and \v{C}ech metric thickenings~\cite{MetricReconstructionViaOptimalTransport}, and sets of positive reach~\cite{FedererCurvature}.

\subsection{Metric spaces}
A \emph{metric space}, $(X,d)$, is a set $X$ along with a function $d \colon X \times X \to \R$ such that for all $u,v,w \in X$,
\begin{itemize}
\item $d(u,v) \ge 0$ and $d(u,v) = 0$ if and only if $u = v$,
\item $d(u,v) = d(v,u)$, and
\item $d(u,v) \le d(u,w) + d(w,v)$.
\end{itemize}
The function $d$ is called a \emph{distance} or a \emph{metric}.
We will denote open balls in $X$ by $\ball{x}{r}=\{y\in X~|~d(y,x)<r\}$, where $x \in X$ is the ball's center and $r$ is its radius.
Likewise, a closed ball will be denoted $\cball{x}{r}=\{y\in X~|~d(y,x)\le r\}$.

Given a point $x\in X$ and subset $Y\subseteq X$, we define the distance between $x$ and $Y$ to be $d(x,Y)=\inf_{y\in Y}d(x,y)$. The distance between two subsets $Y,Y'\subseteq X$ is $d(Y,Y')=\inf_{y\in Y,y'\in Y'}d(y,y')$. An $r$-\emph{thickening} of a metric space $X$ is a metric space $Z \supseteq X$ such that the metric on $X$ extends to that on $Z$, and also $d(x,Z)\le r$ for all $x\in X$. We define the diameter of a set $Y \subseteq X$ to be $\diam(Y) = \sup\{d(y,y')~|~y,y'\in Y\}$.

Any metric $d$ induces a topology on the set $X$ called the metric topology.
The basis for this topology consists of all open balls of positive radius.
A topological space $X$ is \emph{metrizable} if there exists a metric $d\colon X \times X\to \R$ that induces the topology of $X$.

\subsection{Euclidean Space}
Euclidean space is the metric space $(\R^n,d)$ where $d$ is the usual Euclidean distance.
There is a standard inner product $\langle \cdot , \cdot \rangle$ on $\R^n$ defined by
\[
\langle (u_1,\ldots,u_n) , (v_1,\ldots,v_n) \rangle = u_1v_1+\ldots+u_n v_n.
\]
We can define the \emph{norm}, $\|\cdot \|$, of an element $x \in \R^n$ by $\|x\| = \langle x , x \rangle^{1/2}$.
The metric $d$ is then simply $d(u,v)=\|u - v\|$.

For any $X \subseteq \R^n$ define the convex hull of $X$ as
\[\conv(X) = \Bigl\{ x=\sum_{i = 0}^k \lambda_i x_i \in \R^n~\Big|~ k\in\N,\ \lambda_i \geq 0,\ \sum_{i} \lambda_i = 1,\ x_i\in X \Bigr\}.\]
Note that the diameter of $X$ is the same as the diameter of its convex hull. If $X \subseteq \R^n$, we denote by $X^c$ the complement of $X$, that is, $X^c = \R^n \setminus X$, and by $\overline{X}$ the closure of $X$ in $\R^n$.

\subsection{Topological spaces}\label{ssec:Topological}

Let $Y$ and $Z$ be topological spaces.
Two continuous maps $f,g\colon Y\to Z$ are \emph{homotopy equivalent} if there exists a continuous map $H\colon Y\times I\to Z$ with $H(\cdot,0)=f$ and $H(\cdot,1)=g$, where $I=[0,1]$ is the unit interval~\cite{hatcher2002algebraic}.
Two spaces $Y$ and $Z$ are \emph{homotopy equivalent} if there exist continuous maps $f\colon Y\to Z$ and $g\colon Z\to Y$ such that $g\circ f$ is homotopy equivalent to the identity map $\id_Y$ on $Y$, and similarly $f\circ g$ is homotopy equivalent to the identity map $\id_Z$ on $Z$.

\subsection{Wasserstein metric}\label{ssec:Wasserstein}

In this section we describe a way to put a metric on probability Radon measures. The metric has many names: the Wasserstein, Kantorovich, optimal transport, or earth mover's metric. It is known to solve the Monge-Kantorovich problem (see~\cite{vershik2013long}).

Let $X$ be a metric space equipped with a distance function $d\colon X\times X\to\R$. A measure $\mu$ defined on the Borel sets of $X$ is 
\begin{itemize}
\item \emph{inner regular} if $\mu(B)=\sup\{\mu(K)~|~K\subseteq B\mbox{ is compact}\}$ for all Borel sets~$B$,
\item \emph{locally finite} if every point $x\in X$ has a neighborhood $U$ such that $\mu(U)<\infty$,
\item a \emph{Radon measure} if it is both inner regular and locally finite, and
\item a \emph{probability measure} if $\int_Xd\mu = 1$.
\end{itemize}

The following is from~\cite{edwards2011kantorovich,kellerer1982duality}. Let $\cP(X)$ denote the set of probability Radon measures such that for some (and hence all) $y\in X$, we have $\int_X d(x,y)\ d\mu<\infty$.
Define the $L^1$ metric on $X\times X$ by setting the distance between $(x_1,x_2),(x'_1,x'_2)\in X\times X$ to be $d(x_1,x'_1)+d(x_2,x'_2)$.
Given $\mu,\nu\in\cP(X)$, let $\Pi(\mu,\nu)\subseteq\cP(X\times X)$ be the set of all probability Radon measures $\pi$ on $X\times X$ such that $\mu(B)=\pi(B\times X)$ and $\nu(B)=\pi(X\times B)$ for all Borel subsets $B\subseteq X$.
Note that such an element $\pi$ is a joint measure on $X\times X$ whose marginals, when restricted to each $X$ factor, are $\mu$ and $\nu$.
\begin{definition}\label{def:Wasserstein}
The 1-Wasserstein metric on $\cP(X)$ is defined by
\[ d_W(\mu,\nu)=\inf_{\pi\in\Pi(\mu,\nu)}\int_{X\times X}d(x,y)\ d\pi. \]
\end{definition}

The names \emph{optimal transport} or \emph{earth mover's} metric can be interpreted as follows. One can think of measures $\mu$ and $\nu$ as ``piles of dirt" in $X$ with prescribed mass distributions. The joint measure $\pi$ with $\mu$ and $\nu$ as marginals is a transport plan moving the $\mu$ pile of dirt to the $\nu$ pile. The 1-Wasserstein distance between $\mu$ and $\nu$ is the infimum, over all transport plans $\pi$, of the work involved in moving $\mu$ to $\nu$ via transport plan $\pi$.
 
\subsection{Simplicial Complexes}

\begin{definition}\label{def:SimplicialComplex}
Let $V$ be a set, called the set of vertices. An \emph{abstract simplicial complex} $K$ on vertex set $V$ is a subset of the power set of $V$ with the property that if $ \sigma \in K$, then all subsets of $\sigma$ are in $K$.
\end{definition}

Every abstract simplicial complex permits a geometric realization, $|K|$, which is a topological space.
As a set we have
\[ |K|=\Biggl\{\sum_{i=0}^k\lambda_i v_i~\Big|~k\in\N,\ \lambda_i\ge 0,\ \sum_i\lambda_i=1,\ [v_0,\ldots,v_k]\in K\Biggr\}. \]
In the simpler case when $K$ is finite, we can put a topology on $|K|$ as follows. Choose an affinely independent set of points in $\R^n$ (for $n$ sufficiently large) to correspond to each of the elements of the vertex set $V$.
Then $|K|$ consists of all convex linear combinations of these points, and $|K|$ is given its topology as a subset of Euclidean space.
More generally, one can produce a topology on $|K|$ by viewing it as a subset of $[0,1]^V$, the space of functions $V\to[0,1]$.
Indeed note
\[ |K|=\Biggl\{f\colon[0,1]\to V~\Bigg|~\sum_{v\in V}f(v)=1,\ \supp(f)\in K\Biggr\}. \]
Give $[0,1]^V$ its induced topology as the direct limit of $[0,1]^\tau$ where $\tau$ ranges over all finite subsets of $V$, and equip $|K|$ with the subspace topology~\cite{spanier1994algebraic}.

For the rest of this paper we denote both an abstract simplicial complex and its geometric realization by the same symbol $K$.

\subsection{The Vietoris--Rips and \v{C}ech simplicial complexes}

Two natural ways to enlarge a metric space by building a simplicial complex on top of it are the Vietoris--Rips and \v{C}ech complex constructions.

\begin{definition}\label{def:VR}
Let $X$ be a metric space and $r>0$ a scale parameter.
The \emph{Vietoris--Rips simplicial complex} of $X$ with scale parameter $r$, denoted $\vrleq{X}{r}$, has vertex set $X$ and a simplex for every finite subset $\sigma \subseteq X$ such that $\diam(\sigma) \leq r$. 
Similarly, $\vrless{X}{r}$ contains every finite subset with diameter $< r$.
\end{definition}

\begin{figure}[h]
\begin{centering}
\def\svgwidth{5in}
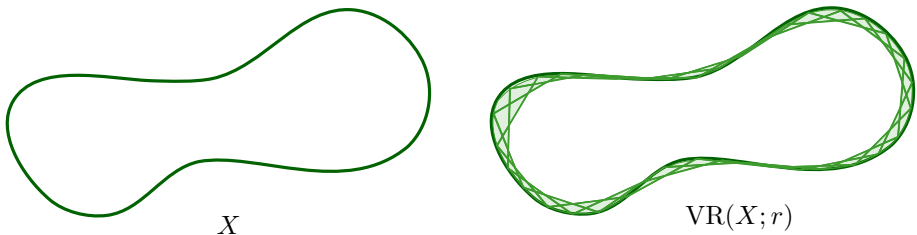
\caption{A metric space $X$ and (a subset of) its Vietoris--Rips complex.}
\end{centering}
\end{figure}

We will write $\vr{X}{r}$ when the distinction between $<$ and $\le$ is unimportant.
The Vietoris--Rips complex is the \emph{clique} or \emph{flag} complex of its 1-skeleton.

\begin{definition}\label{def:ambient-cech}
Let $X\subseteq Y$ be a submetric space and $r$ a scale parameter with $r \geq 0$.
The \v{C}ech complex of $X$ with scale parameter $r$, $\cechaleq{X}{Y}{r}$, has vertex set $X$ and a simplex for every finite subset
$\sigma \subseteq X$ such that $\bigcap_{x_i \in \sigma} \cball{x_i}{r/2} \neq \emptyset$,
where $\cball{x_i}{r/2}$ denotes a closed ball in $Y$ centered at $x_i$ with radius $r/2$.

Similarly, $\cechaless{X}{Y}{r}$ contains a simplex for every finite subset $\sigma$ such that 
$\bigcap_{x_i \in \sigma} \ball{x_i}{r/2} \neq \emptyset$.
\end{definition}

Again, we will write $\cecha{X}{Y}{r}$ when the distinction between open and closed is unimportant.
The \v{C}ech complex can be considered as the nerve of the union of balls in $Y$ of radius $r/2$ centered at each of the points in $X$.
Of particular interest are the cases where $Y = \R^n$ and $Y = X$.
These are called the called the \emph{ambient} and \emph{intrinsic} \v{C}ech complex, respectively.
Note that if $X \subseteq\R^n$ then $\cecha{X}{X}{r} \subseteq \cecha{X}{\R^n}{r}$.
When it is not necessary to distinguish these two we will write $\cech{X}{r}$.

For any $\sigma\in\cecha{X}{Y}{r}$ we have $\diam(\sigma) \le r$, and so $\cecha{X}{Y}{r}$ is a subcomplex of $\vr{X}{r}$.
When $Y$ is a geodesic space, the complexes $\vr{X}{r}$ and $\cecha{X}{Y}{r}$ have the same $1$-skeletons.
But independent of whether $Y$ is geodesic or not, the \v{C}ech complex can be a proper subset of the Vietoris--Rips complex.

A useful characterisation is the following:
\begin{proposition}\label{prop:cech-ball}
A set of points $x_0 , \ldots , x_k\in X$ form a simplex in $\cechaless{X}{Y}{r}$ if and only if there exists a point $c\in Y$ such that every $x_i$ is contained in the open ball $\ball{c}{r/2}$. A similar statement is true for $\cechaleq{X}{Y}{r}$ with closed balls.
\end{proposition}

\begin{proof}
Let $c$ be any point in the intersection $\bigcap_{x_i \in \sigma} B(x_i,r/2)$.
Then $B(c,r/2)$ contains all of the $x_i$.
Conversely, if all of the $x_i$ are in $B(c,r/2)$, then $c$ is in $\bigcap_{x_i \in \sigma} B(x_i,r/2)$ and hence the intersection is nonempty.
\end{proof}

The Vietoris--Rips and \v{C}ech complexes are given the standard topology as simplicial complexes: a subset of the geometric realization is open if and only if its intersection with every simplex is open.
An important remark is the following. 
A simplicial complex $K$ is said to be \emph{locally finite} if each vertex belongs to only finitely many simplices of $K$, and a simplicial complex is metrizable if and only if it is locally finite~\cite[Proposition~4.2.16(2)]{sakai2013geometric}. This means that in general, the Vietoris--Rips and \v{C}ech simplicial complexes cannot be equipped with a metric without changing their homeomorphism types, even though they were built on top of a metric space.

\subsection{The Vietoris--Rips and \v{C}ech Thickenings}

The definitions in this section are from~\cite{MetricReconstructionViaOptimalTransport,edwards2011kantorovich,kellerer1982duality}.
 
Given a metric space $X$ and a scale parameter $r$ we will define the \emph{Vietoris--Rips thickening} $\vrm{X}{r}$, which will be a metric space $r$-thickening of $X$. As a set, $\vrm{X}{r}$ is the set of all formal convex combinations of points in $X$ with diameter at most $r$, namely\footnote{There is a canonical bijection between the sets underlying $\vrm{X}{r}$ and the geometric realization $|\vr{X}{r}|$ of the Vietoris--Rips simplicial complex. However, these two topological spaces will often not be homeomorphic (and sometimes not even homotopy equivalent).}
\begin{align*}
\vrmleq{X}{r}&=\Bigl\{ \sum_{i = 0}^k \lambda_i x_i~|~k\in\N,\ \lambda_i\ge 0,\ \sum_i\lambda_i=1,\ x_i\in X\mbox{, and }\diam(\{x_0,\ldots,x_k\}) \le r \Bigr\}\\
\vrmless{X}{r}&=\Bigl\{ \sum_{i = 0}^k \lambda_i x_i~|~k\in\N,\ \lambda_i\ge 0,\ \sum_i\lambda_i=1,\ x_i\in X\mbox{, and }\diam(\{x_0,\ldots,x_k\}) < r \Bigr\}.
\end{align*}
A useful viewpoint is to consider an element of $\vrm{X}{r}$ as a probability measure.
For $x\in X$, let $\delta_x$ be the Dirac probability measure defined on any Borel subset $E \subseteq X$ by
\[
   \delta_x(E) = \left\{\begin{array}{l c} 1 & \mbox{if }x \in E \\
   0 & \mbox{if }x \notin E.
   \end{array}\right.
\]
By identifying $x \in X$ with $\delta_x\in\cP(X)$, and more generally $x=\sum_{i = 0}^k \lambda_i x_i$ with $ \sum_{i = 0}^k \lambda_i \delta_{x_i}\in\cP(X)$, we can view $\vrm{X}{r}$ as a subset of $\cP(X)$, the set of all Radon probability measures on $X$.
Hence we can equip the set $\vrm{X}{r}$ with the $1$-Wasserstein metric from Section~\ref{ssec:Wasserstein}, namely $d_W(x,x') = \inf_{\pi \in \Pi(x,x')} \int_{X \times X} d(x,x')\ d\pi$ for $x,x'\in\vrm{X}{r}$.

To give a more explicit definition of the metric on $\vrm{X}{r}$, let $x,x' \in \vrm{X}{r}$ with $x = \sum_{i = 0}^k \lambda_i x_i$ and $x' = \sum_{i = 0}^{k'} \lambda_i' x'_i$ (we cease to distinguish between $x \in X$ and its associated measure, $\delta_x$).
Define a matching $p$ between $x$ and $x'$ to be any collection of non-negative real numbers $\{p_{i,j}\}$ such that $\sum_{j = 0}^{k'} p_{i,j} = \lambda_i $ and $\sum_{i = 0}^k p_{i,j} = \lambda_j'$. It follows as a consequence that $\sum_{i,j}p_{i,j}=1$, and so matching $\{p_{i,j}\}$ can be thought of as a joint probability distribution with marginals $\{\lambda_i\}_{i=0}^k$ and $\{\lambda'_j\}_{j=0}^{k'}$. 
Define the \emph{cost of the matching $p$} to be $\cost(p) = \sum_{i,j} p_{i,j} d(x_i,x_j')$.

\begin{definition}\label{def:Wasserstein}
The $1$-Wasserstein metric on $\vrm{X}{r}$ is the distance $d_W$ defined by 
\begin{equation*}
d_W(x,x') = \inf \left\{ \cost(p) \suchthat p \text{ is a matching between } x \text{ and } x' \right\}.
\end{equation*}
\end{definition}

Similar to the Vietoris--Rips thickening, we can construct the \v{C}ech thickening $\cecham{X}{Y}{r}$ equipped with the 1-Wasserstein metric.
The construction is exactly the same, except that the elements of $\cecham{X}{Y}{r}$ are the convex combinations of vertices from simplices in $\cecha{X}{Y}{r}$ (rather than in $\vr{X}{r}$).
By~\cite[Lemma~3.5]{MetricReconstructionViaOptimalTransport}, both $\vrm{X}{r}$ and $\cecham{X}{Y}{r}$ are $r$-thickenings of the metric space $X$.

One could alternatively consider a $p$-Wasserstein metric for $1\le p\le \infty$.

\subsection{Sets of positive reach}
We are interested in the case where metric space $X$ is a subset of $\R^n$ of \emph{positive reach}.
In particular, any embedded $C^k$ submanifold (with or without boundary) of $\R^n$ with $k \ge 2$ has positive reach~\cite{Thale}, but not every set of positive reach is a manifold~\cite{FedererCurvature,rataj2017structure}.
Consider the set 
\[
Y = \left\{ y \in \R^n \suchthat \exists\, x_1 \neq x_2 \in X \text{ with } d(y,x_1) = d(y,x_2) = d(y,X) \right\}.
\]
The closure $\overline{Y}$ of $Y$ is the \emph{medial axis} of $X$.
For any point $x \in X$, the \emph{local feature size at $x$} is the distance $d(x,\overline{Y})$ from $x$ to the medial axis.
The \emph{reach} $\tau$ of $X$ is the minimal distance $\tau = d(X,\overline{Y})$ between $X$ and its medial axis.

For $X \subseteq\R^n$ and $\alpha > 0$ we define its $\alpha$-offset (or tubular neighborhood), $\Tub_{\alpha}$, by 
\[
\Tub_{\alpha} = \left\{ x \in \R^n \suchthat d(x,X) < \alpha \right\} = \bigcup_{x \in X} \ball{x}{\alpha} .
\]
In particular, if $X$ has reach $\tau$, then for every point in $\Tub_{\tau}$ there exists a unique nearest point in $X$.
As in~\cite{FedererCurvature,NiyogiSmaleWeinberger}, define $\pi \colon \Tub_{\tau} \to X$ to be the nearest point projection map, sending an element $x\in\Tub_{\tau}$ to its unique closest point $\pi(x)\in X$.

\begin{lemma}\label{lem:pi-continuous}
The function $\pi \colon \Tub_{\tau} \to X$ is continuous.
\end{lemma}

\begin{proof}
Let $x,y \in \Tub_{\tau}$ and $r = \max\{d(x,\pi(x)),d(y,\pi(y))\}$.
Then the conditions of~\cite[Theorem~4.8(8)]{FedererCurvature} are satisfied and so
\begin{equation}
d(\pi(x),\pi(y)) \leq \frac{\tau}{\tau - r} d(x,y).
\end{equation}
Thus $\pi$ is continuous at $x$ for any $x \in \Tub_{\tau}$.
\end{proof}

We also state the following proposition, implicit in~\cite{NiyogiSmaleWeinberger}, for any set of positive reach.

\begin{proposition}\label{prop:empty-ball}
Let $X \subseteq\R^n$ have reach $\tau > 0$.
Let $p \in X$ and suppose $x\in\Tub_\tau\setminus X$ satisfies $\pi(x)=p$.
If $c=p+\tau\frac{x-p}{\|x-p\|}$, then $\ball{c}{\tau} \cap X = \emptyset$.
\end{proposition}

\begin{proof}
For any $0 < t < \tau$, let $y_t = p + t\frac{x-p}{\|x-p\|}$.
Since $y_t \in \Tub_\tau$, we have $\cball{y_t}{t} \cap X = \{p\}$ and $d(y_t,p) = t$, so $\ball{c}{t} \cap X = \emptyset$.
Note that $\ball{c}{\tau} = \cup_{0 < t < \tau} \ball{y_t}{t}$. Indeed, to see the inclusion $\subseteq$, suppose that $z \in \ball{c}{\tau}$, so that $d(z,c) = \tau - \epsilon$ for some $\epsilon > 0$.
Let $t = \tau - \frac{\epsilon}{3}$.
By the triangle inequality, $d(y_t,z) \le d(y_t,c) + d(c,z) = \tau - \frac{2 \epsilon}{3} < t$, giving $z\in \ball{y_t}{t}$.
The reverse inclusion $\supseteq$ is straightforward.
It follows that $\ball{c}{\tau} \cap X = \emptyset$.
\end{proof}

\section{A metric and Euclidean analogue of Hausmann's result}\label{sec:Euclidean-Hausmann}

We now present our main theorem, a metric analogue of Hausmann's result for Vietoris--Rips thickenings of subsets of Euclidean space with positive reach.
Since in Section~\ref{sec:Euclidean-nerve} we will also give an analogous theorem for the metric \v{C}ech thickening, we provide some notation now for both cases.
Let $X \subseteq \R^n$ be a set of positive reach.
Let $K(X;r)$ be either a Vietoris--Rips complex or \v{C}ech complex of $X$ with scale parameter $r$, and let $K^m(X;r)$ be the corresponding metric Vietoris--Rips or \v{C}ech thickening.
Define $f\colon K^m(X;r) \rightarrow \R^n$ to be the linear projection map $f\left( \sum_{i} \lambda_i x_i \right) = \sum_{i} \lambda_i x_i \in \R^n$, where the first sum is a formal convex combination of points in $X$, and the second sum is the standard addition of vectors in $\R^n$.
Recall $\pi \colon \Tub_{\tau} \to X \subseteq \R^n$ is the nearest-point projection map.

Several geometric lemmas are required.

\begin{lemma}\label{lem:convex}
Let $x_0,\ldots,x_k\in\R^n$, let $y\in\conv\{x_0,\ldots,x_k\}$, and let $C$ be a convex set with $y\notin C$. Then there is at least one $x_i$ with $x_i\notin C$.
\end{lemma}

\begin{proof}
Suppose for a contradiction that we had $x_i\in C$ for all $i=0,\ldots,k$. Then since $C$ is convex, we'd also have $y\in \conv\{x_0,\ldots,x_k\}\subseteq C$. Hence it must be the case that $x_i\notin C$ for some $i$.
\end{proof}

\begin{lemma}\label{lem:vr-tub}
For $X\subseteq\R^n$ and $r>0$, the map $f\colon \vrm{X}{r}\to\R^n$ has its image contained in $\overline{\Tub_r}$.
\end{lemma}

\begin{proof}
Let $x = \sum_{i=0}^k \lambda_i x_i\in \vrm{X}{r}$; we have
\[ \diam(\conv\{x_0,\ldots,x_k\})=\diam([x_0 , \ldots , x_k])\le r.\]
Since $f(x)\in\conv\{x_0,\ldots,x_k\}$, it follows that $d(f(x),X)\le d(f(x),x_0)\le r$, and so $f(x)\in\overline{\Tub_r}$.
\end{proof}

The substance of Lemma~\ref{lem:vr-simplex} will be that if $[x_0,\ldots,x_k]$ is a simplex in $\vr{X}{r}$ and if $x=\sum_i\lambda_i x_i\in \vrm{X}{r}$, then $\pi(f(x))$ will be ``close enough" to $x_0,\ldots,x_k$ so that $[x_0,\ldots,x_k,\pi(f(x))]$ is also a simplex in $\vr{X}{r}$.
This fact will be crucial for defining the homotopy equivalences in our proof of Theorem~\ref{thm:main}.

\begin{lemma}\label{lem:vr-simplex}
Let $X\subseteq\R^n$ have positive reach $\tau$, let $[x_0, \ldots x_k]$ be a simplex in $\vr{X}{r}$ with $r < \tau$, let $x = \sum \lambda_i x_i\in\vrm{X}{r}$, and let $p=\pi(f(x))$.
Then the simplex $[x_0 , \ldots , x_k , p]$ is in $\vr{X}{r}$.
\end{lemma}
\begin{proof}
We write the proof for $\vrleq{X}{r}$; an analogous proof works for $\vrless{X}{r}$. 
Note $p=\pi(f(x))$ is defined by Lemma~\ref{lem:vr-tub} since $\overline{\Tub_r}\subseteq \Tub_\tau$.
We may assume $p\neq f(x)$, since otherwise the conclusion follows as $f(x)$ is in the convex hull of the $x_i$.

Suppose for a contradiction that $d(x_i, p)> r$ for some $i$; without loss of generality we may assume $i=0$.
Since $d(x_0, f(x))\le r$ we have that $f(x)\neq p$.
Following~\cite{NiyogiSmaleWeinberger}, let $c=p+\tau\frac{f(x)-p}{\|f(x)-p\|}$, and let $\ball{c}{\tau}$ be the open ball of radius $\tau$ that is tangent to $X$ at $p$.
By Proposition~\ref{prop:empty-ball} this open ball does not intersect $X$, giving $x_0,\ldots,x_k\notin \ball{c}{\tau}$.
Define $T_p^{\perp}$ to be the line through $f(x)$ and $p$.
Since $f(x)$ is between $p$ and $c$ on $T_p^\perp$, note that $d(x_0, f(x))\le r$ implies $x_0$ is not on $T_p^\perp$.
Let $x'_0\neq x_0$ be the closest point on $T_p^\perp$ to $x_0$. Let $H_{x_0}=\{z\in\R^n~|~\langle z-x'_0,x_0-x'_0\rangle>0\}$ be the open half-space containing $x_0$, whose boundary is the hyperplane containing $T_p^\perp$ that's perpendicular to $x_0-x'_0$. 
Since $d(x_0, p),d(x_0,c)>r$, it follows that $H_{x_0}^c\cap \cball{x_0}{r}\subseteq \ball{c}{\tau}$.
Since $x_i\in \cball{x_0}{r}\setminus \ball{c}{\tau}$, this implies that $x_i\in H_{x_0}$ for all $i$.
This contradicts Lemma~\ref{lem:convex} since $H_{x_0}$ is convex with $f(x)\notin H_{x_0}$, even though $f(x)\in\conv(\{x_0,\ldots,x_k\})$.
Hence it must be the case that $d(x_0,p)\le r$, and it follows that $[x_0 , \ldots , x_k , p]$ is a simplex in $\vrleq{X}{r}$.
\end{proof}

\begin{figure}[h]
\def\svgwidth{3in}
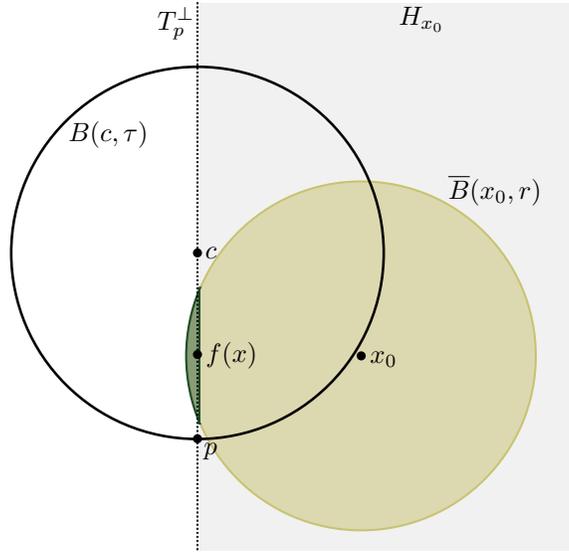
\caption{Figure for the proof of Lemma~\ref{lem:vr-simplex}. The green shaded region is a subset of $\ball{c}{\tau}$, forcing all $x_i$ to be in the same half-space.}
\end{figure}

We are now prepared to prove our main result. We remark that when $X$ is a manifold, it may be possible to estimate the reach of $X$ using the techniques of~\cite{aamari2017estimating,attali2009stability,dey2006normal,fefferman2016testing}.

\begin{theorem}\label{thm:main}
Let $X$ be a subset of Euclidean space $\R^n$, equipped with the Euclidean metric, and suppose the reach $\tau$ of $X$ is positive. Then for all $r < \tau$, the metric Vietoris--Rips thickening $\vrm{X}{r}$ is homotopy equivalent to $X$.
\end{theorem}

\begin{figure}[h]
\def\svgwidth{6in}
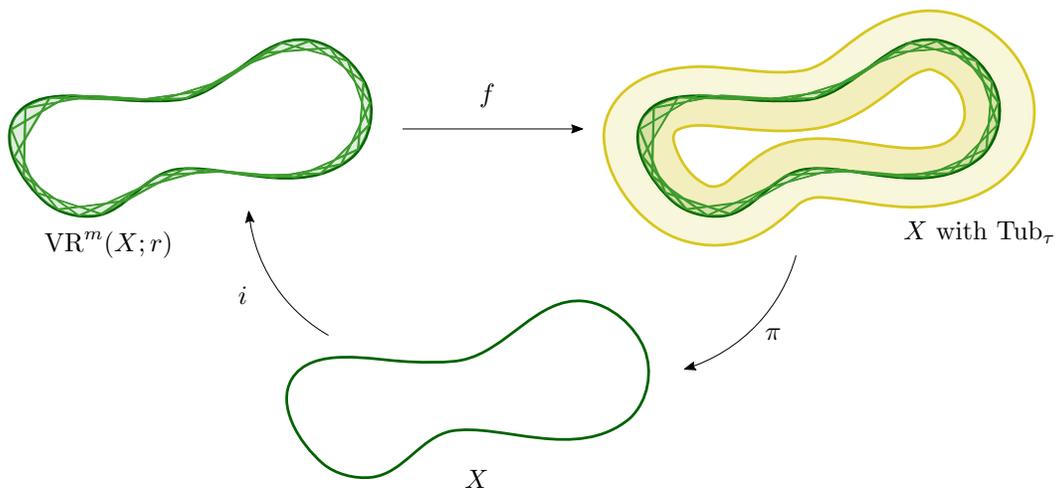
\caption{The homotopy equivalence between $\vrm{X}{r}$ and $X$ in Theorem~\ref{thm:main}.}
\end{figure}

\begin{proof}
By~\cite[Lemma~5.2]{MetricReconstructionViaOptimalTransport}, map $f\colon \vrm{X}{r}\to \R^n$ is $1$-Lipschitz and hence continuous.
It follows from Lemma~\ref{lem:vr-tub} that the image of $f$ is a subset of $\Tub_\tau$.
By Lemma~\ref{lem:pi-continuous} we have that $\pi \colon \Tub_{\tau} \to X$ is continuous.
Let $\iota \colon X \rightarrow \vrm{X}{r}$ be the natural inclusion map.

We will show that $\iota$ and $\pi \circ f$ are homotopy inverses.
Note that $\pi \comp f \comp i = \id_{X}$.
Consider the map $H\colon \vrm{X}{r} \times I \rightarrow \vrm{X}{r}$ defined by $H(x,t) = t\cdot\id_{\vrm{X}{r}} + (1 - t) \iota \comp \pi \comp f$.
Map $H$ is well-defined by Lemma~\ref{lem:vr-simplex}, and continuous by~\cite[Lemma~3.8]{MetricReconstructionViaOptimalTransport} since $\pi \comp f$ is continuous. 
It follows that $H$ is a homotopy equivalence from $\iota \circ \pi \circ f$ to $id_{\vrm{X}{r}}$, and hence $\vrm{X}{r}$ is homotopy equivalent to $X$.
\end{proof}

\section{A metric and Euclidean analogue of the nerve lemma}\label{sec:Euclidean-nerve}

We handle the case of \v{C}ech thickenings in a similar fashion in this section. We write $\cechm{X}{r}$ for either the ambient \v{C}ech complex $\cecham{X}{\R^n}{r}$ or the intrinsic \v{C}ech complex $\cecham{X}{X}{r}$ when the distinction is not important.

\begin{lemma}\label{lem:cech-radius}
Let $[x_0 , \ldots , x_k]$ be a simplex in $\cech{X}{2r}$, where $X\subseteq\R^n$.
Then for any $x \in \conv([x_0 , \ldots , x_k])$, there exists a vertex $x_i$ such that $d(x,x_i) \le r$.
\end{lemma}
\begin{proof}
We follow the proof of~\cite[Lemma~2.9]{DeSilvaGhrist} closely.
By assumption, balls of radius $r$ centered at the points $x_i$ meet at a common point $y$.
Let $x = \sum_i \lambda_i x_i$ be a point in $\conv([x_0 , \ldots , x_k])$.
Rewrite this as $\vec{0} = \lambda_0 \hat{x}_0 + \cdots + \lambda_k \hat{x}_k$ where $\hat{x}_i = x_i - x$.
Also let $\hat{y} = y - x$.
Taking the dot product with $\hat{y}$ gives
\[
   0 = \lambda_0\langle\hat{x}_0,\hat{y}\rangle + \cdots + \lambda_k\langle\hat{x}_k,\hat{y}\rangle.
\]
So for some $i$ we have $\langle\hat{x}_i,\hat{y}\rangle \le 0$.
In that case,
\[
r^2 \ge d(x_i,y)^2 = d(\hat{x}_i,\hat{y})^2 = \|\hat{x}_i\|^2 - 2\langle\hat{x}_i,\hat{y}\rangle + \|\hat{y}\|^2 \ge \|\hat{x}_i\|^2 = \|x_i - x\|^2.
\]
\end{proof}

\begin{lemma}\label{lem:cech-tub}
For $X\subseteq\R^n$ and $r>0$, the map $f\colon \cechm{X}{2r}\to\R^n$ has its image contained in $\overline{\Tub_r}$.
\end{lemma}

\begin{proof}
For any point $x=\sum_i\lambda_i x_i\in \cechm{X}{2r}$ we have that $f(x) \in \conv(\{x_0 , \ldots , x_k\})$.
The result then follows form Lemma~\ref{lem:cech-radius}.
\end{proof}

\begin{lemma}\label{lem:cech-simplex-ambient}
Let $X\subseteq\R^n$ have positive reach $\tau$, let $[x_0, \ldots x_k]$ be a simplex in $\cecha{X}{\R^n}{2r}$ with $r < \tau$, let $x = \sum \lambda_i x_i\in\cecham{X}{\R^n}{2r}$, and let $p=\pi(f(x))$.
Then the simplex $[x_0 , \ldots , x_k , p]$ is in $\cecha{X}{\R^n}{2r}$.
\end{lemma}

\begin{proof}
We write the proof for $\cechaleq{X}{\R^n}{2r}$; an analogous proof works for $\cechaless{X}{\R^n}{2r}$. 
Since $[x_0 , \ldots , x_k]$ is a simplex in $\cechaless{X}{\R^n}{2r}$, there exists a ball $\cball{y}{r}$ of radius $r$ centered at some point $y\in\R^n$ such that $x_i \in \cball{y}{r}$ for all $i$.
Also note that $f(x) \in \conv\{x_0,\ldots,x_k\}\subseteq \cball{y}{r}$.
Note $p=\pi(f(x))$ is defined by Lemma~\ref{lem:cech-tub}.
We may assume $p\neq f(x)$, since otherwise the conclusion follows from $f(x)\in \cball{y}{r}$.
Similarly, we know that $d(p,f(x)) \le r$ since $d(x_i,f(x)) \le r$ for some $i$ and since $p$ is the closest point in $X$ to $f(x)$.

Suppose for a contradiction that $p \notin \cball{y}{r}$.
Let $c=p+\tau\frac{f(x)-p}{\|f(x)-p\|}$, and let $\ball{c}{\tau}$ be the open ball with center $c$ and radius $\tau$ that is tangent to $X$ at $p$.
By Proposition~\ref{prop:empty-ball} every $x_i$ must be in $\cball{y}{r} \setminus \ball{c}{\tau}$.
Let $T_p^{\perp}$ be the line through $f(x)$ and $p$.
We claim that $y$ cannot lie on $T_p^{\perp}$.
Indeed if $y$ were on $T_p^\perp$ its location would be limited to one of three line segments --- one with $p$ between $y$ and $f(x)$, one with $y$ between $p$ and $f(x)$, and one with $f(x)$ between $p$ and $y$.
The first cannot occur as we would have $d(p,y) \le d(f(x),y) \le r$.
The second cannot occur as we would have $d(p,y) \le d(p,f(x)) \le r$.
Finally, the third cannot occur because either $d(p,y) < 2\tau - r$ and so the ball $\cball{y}{r}$ is contained in $\ball{c}{\tau}$ and thus cannot contain any vertex $x_i$ in contradiction of the definition of $y$, or $d(p,y) \ge 2\tau - r$, in which case $d(f(x),\cball{y}{r} \setminus \ball{c}{\tau}) > r$ which contradicts Lemma~\ref{lem:cech-radius}.

\begin{figure}[h]
\def\svgwidth{1.8in}
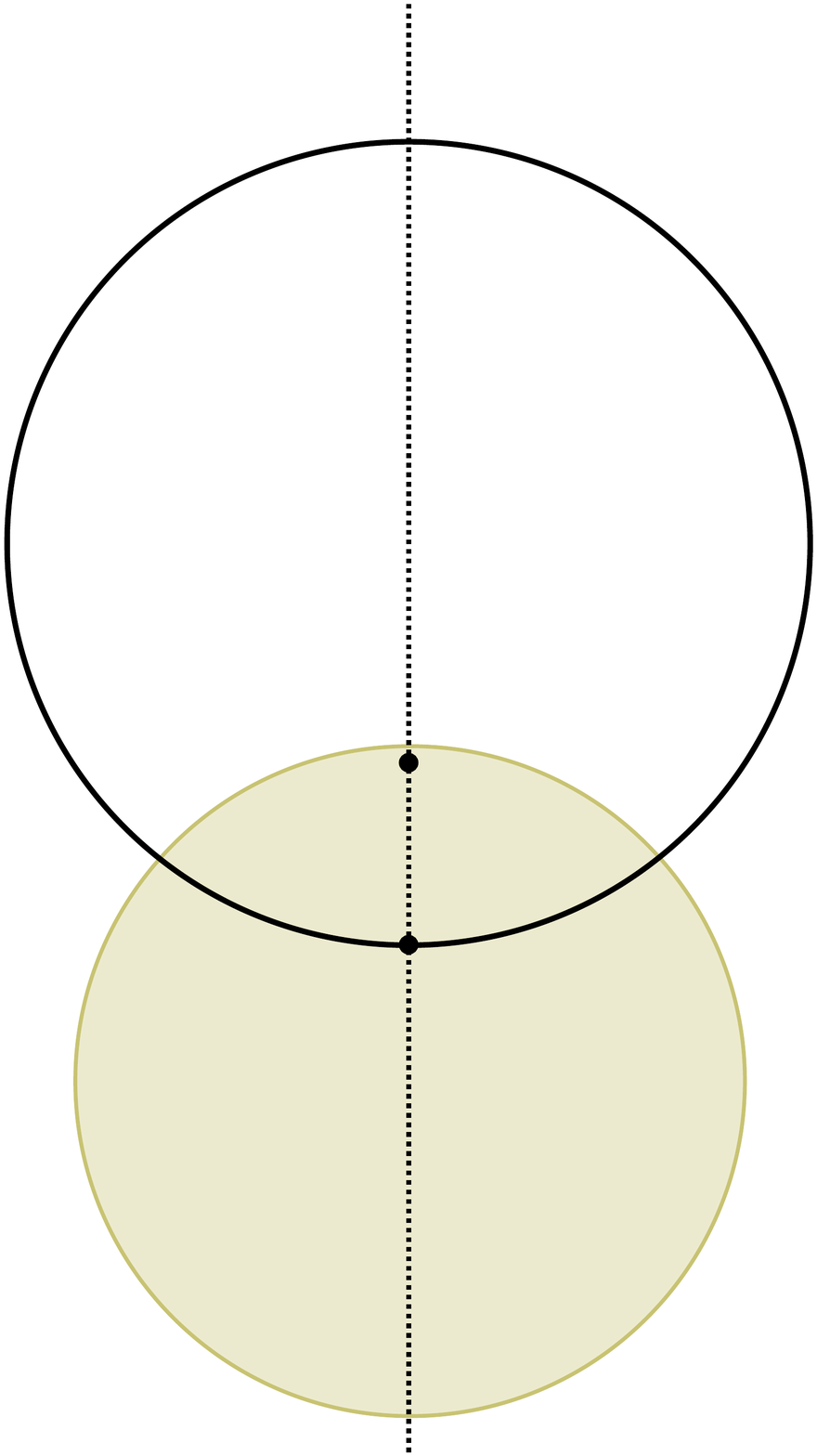
\hspace{5mm}
\def\svgwidth{1.8in}
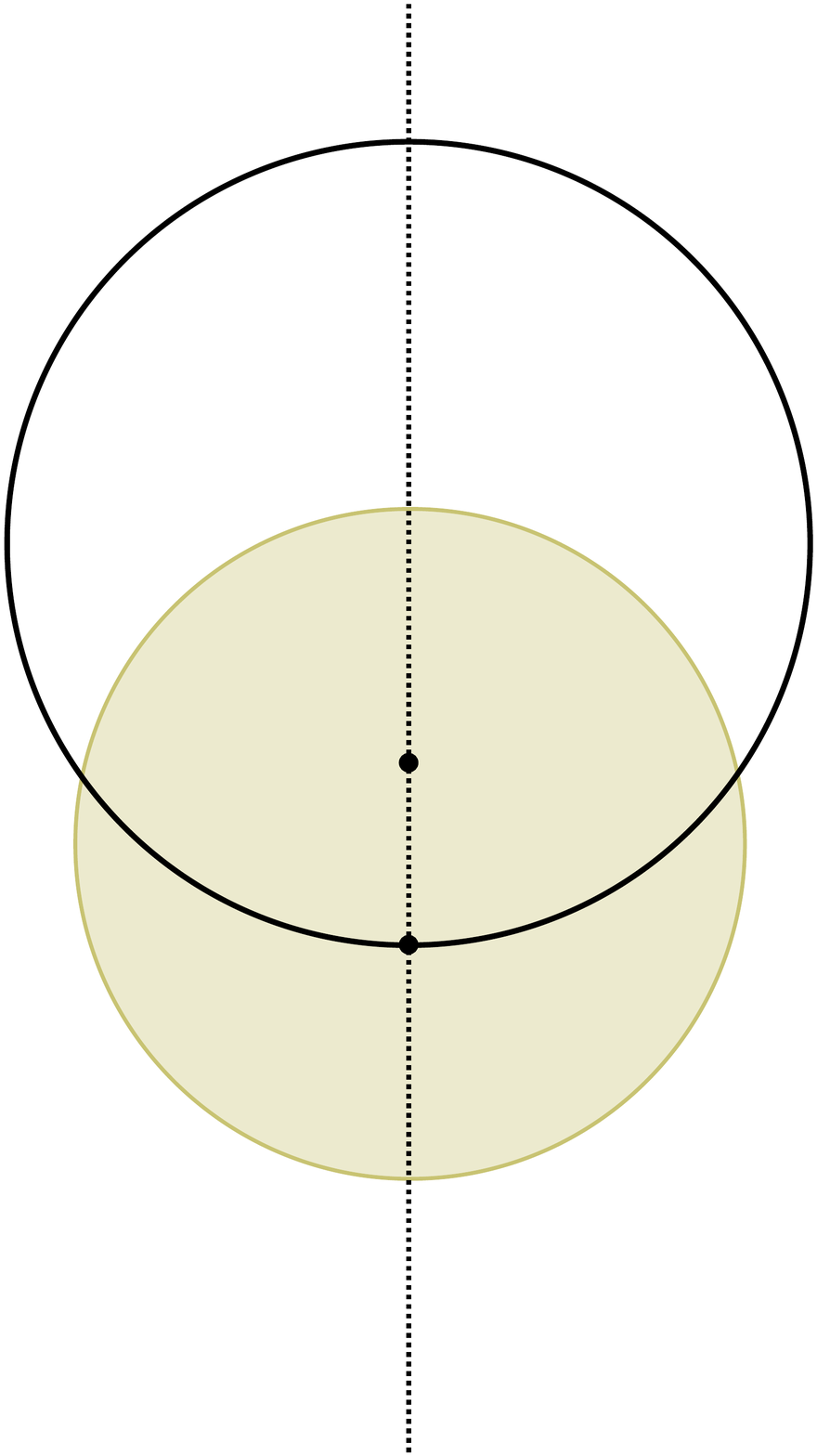
\hspace{5mm}
\def\svgwidth{1.8in}
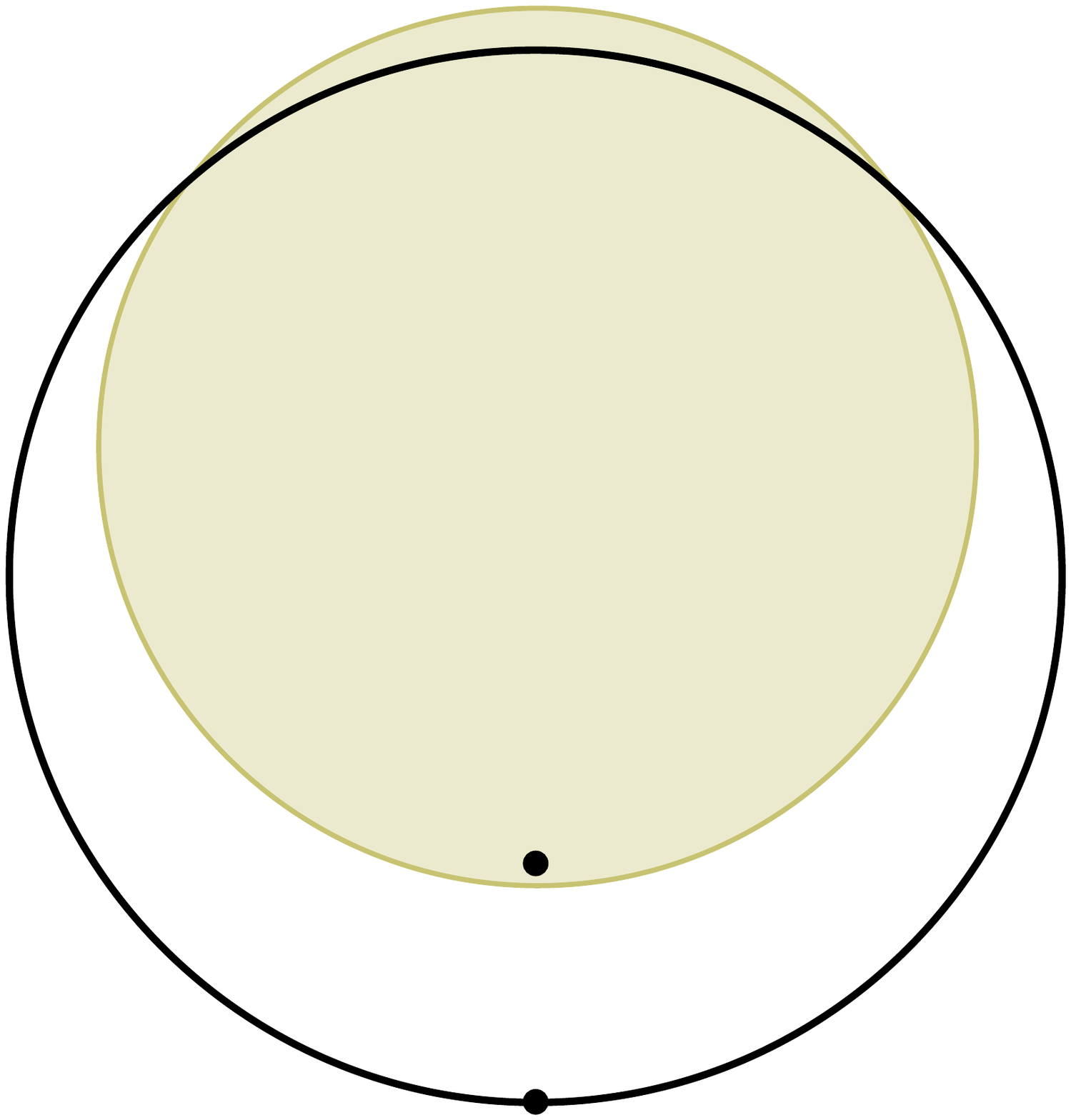
\caption{(\emph{Left}) If $p$ is between $y$ and $f(x)$, then $\cball{y}{r}$ contains $p$.
(\emph{Middle}) If $y$ is between $p$ and $f(x)$ then again $\cball{y}{r}$ contains $p$.
(\emph{Right}) If $f(x)$ is between $y$ and $p$, then the green region $\cball{y}{r} \setminus \ball{c}{\tau}$ is either empty or too far from $f(x)$.}
\end{figure}

\begin{figure}[h]
\def\svgwidth{2in}
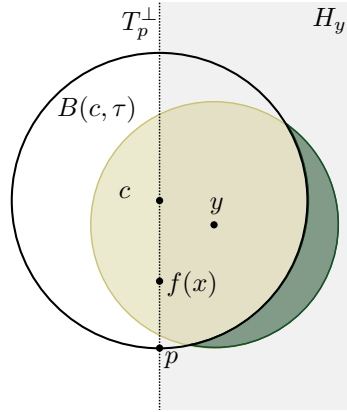
\caption{Figure for the proof of Lemma~\ref{lem:cech-simplex-ambient}. The green region $\cball{y}{r}\setminus\cball{c}{\tau}$ is entirely contained in $H_y$.}
\end{figure}

Let $y'\neq y$ be the closest point on $T_p^\perp$ to $y$.
Let $H_y=\{z\in\R^n~|~\langle z-y',y-y'\rangle>0\}$ be the open half-space containing $y$, whose boundary is the hyperplane containing $T_p^\perp$ that's perpendicular to $y-y'$.
Since $f(x)\in\ball{y}{r}$ and $p\notin\ball{y}{r}$, we have $\cball{y}{r} \setminus \ball{c}{\tau} \subseteq H_y$, which implies $x_i \in H_y$ for all $i$.
This contradicts Lemma~\ref{lem:convex} since $H_y$ is convex with $f(x)\notin H_y$, even though $f(x)\in\conv(\{x_0,\ldots,x_k\})$.
Hence it must be the case that $p \in \cball{y}{r}$, and so $[x_0 , \ldots , x_k , p ]$ is a simplex in $\cechaleq{X}{\R^n}{2r}$.
\end{proof}

\FloatBarrier

An analogous lemma holds for the intrinsic \v{C}ech complex.

\begin{lemma}\label{lem:cech-simplex-intrinsic}
Let $X\subseteq\R^n$ have positive reach $\tau$, let $[x_0, \ldots x_k]$ be a simplex in $\cecha{X}{X}{2r}$ with $r < \tau$, let $x = \sum \lambda_i x_i\in\cecham{X}{X}{2r}$, and let $p=\pi(f(x))$.
Then the simplex $[x_0 , \ldots , x_k , p]$ is in $\cecha{X}{X}{2r}$.
\end{lemma}

\begin{proof}
As in the ambient case, we write the proof for $\cechaleq{X}{X}{2r}$; an analogous proof works for $\cechaless{X}{X}{2r}$.

Since $[x_0 , \ldots , x_k]$ is a simplex in $\cecha{X}{X}{2r}$, there exists a ball $\cball{y}{r}$ of radius $r$ centered at some point $y\in X$ such that $x_i \in \cball{y}{r} \cap X$ for all $i$.
Also note that $f(x) \in \conv\{x_0,\ldots,x_k\}\subseteq \cball{y}{r}$, and again $p=\pi(f(x))$ is well-defined by Lemma~\ref{lem:cech-tub}.
We may assume $p\neq f(x)$, since otherwise the conclusion follows trivially because $p \in X$ and $f(x)\in \cball{y}{r}$, so we would have $p \in \cball{y}{r} \cap X$.
Also, we know that $d(p,f(x)) < r$ since $d(x_i,f(x)) \le r$ for some $i$ and since $p$ is the closest point in $X$ to $f(x)$.

Suppose for a contradiction that $p \notin \cball{y}{r}$.
Let $c=p+\tau\frac{f(x)-p}{\|f(x)-p\|}$, and let $\ball{c}{\tau}$ be the open ball with center $c$ and radius $\tau$ that is tangent to $X$ at $p$.
By Proposition~\ref{prop:empty-ball} every $x_i$ must be in $\cball{y}{r} \setminus \ball{c}{\tau}$.
As above, let $T_p^{\perp}$ be the line through $f(x)$ and $p$.

We now claim that $y$ cannot lie on $T_p^{\perp}$.
Indeed, since $y \in X$, we would have either $y = p$ contradicting $p \notin \cball{y}{r}$, or $d(y,f(x)) > \tau$ because $y \notin \ball{c}{\tau}$, contradicting $f(x) \in \cball{y}{r}$.

Let $y'\neq y$ be the closest point on $T_p^\perp$ to $y$.
Let $H_y=\{z\in\R^n~|~\langle z-y',y-y'\rangle>0\}$ be the open half-space containing $y$, whose boundary is the hyperplane containing $T_p^\perp$ that's perpendicular to $y-y'$.
Since $f(x)\in\ball{y}{r}$ and $p\notin\ball{y}{r}$, we have $\cball{y}{r} \setminus \ball{c}{\tau} \subseteq H_y$, which implies $x_i \in H_y$ for all $i$.
This contradicts Lemma~\ref{lem:convex} since $H_y$ is convex with $f(x)\notin H_y$, even though $f(x)\in\conv(\{x_0,\ldots,x_k\})$.
Hence it must be the case that $p \in \cball{y}{r} \cap {X}$, and so $[x_0 , \ldots , x_k , p ]$ is a simplex in $\cecha{X}{X}{2r}$.
\end{proof}

The following result is related to the nerve lemma, but it is not a consequence thereof. Indeed, even though the \v{C}ech simplicial complex $\cech{X}{2r}$ is the nerve of a collection of balls, the metric \v{C}ech thickening $\cechm{X}{2r}$ in general need not be homeomorphic nor even homotopy equivalent to the nerve $\cech{X}{2r}$. We reiterate that this theorem is for both the ambient and intrinsic \v{C}ech complexes, $\cecham{X}{\R^n}{2r}$ and $\cecham{X}{X}{2r}$.

\begin{theorem}\label{thm:cech}
Let $X$ be a subset of Euclidean space $\R^n$, equipped with the Euclidean metric, and suppose the reach $\tau$ of $X$ is positive. Then for all $r < \tau$, the metric \v{C}ech thickening $\cechm{X}{2r}$ is homotopy equivalent to $X$.
\end{theorem}

\begin{proof}
We follow the same outline as for the Vietoris--Rips case.
Map $f\colon \cechm{X}{2r}\to \R^n$ is again continuous by~\cite[Lemma~5.2]{MetricReconstructionViaOptimalTransport}.
It follows from Lemma~\ref{lem:cech-tub} that the image of $f$ is a subset of $\Tub_\tau$.
By Lemma~\ref{lem:pi-continuous} we have that $\pi \colon \Tub_{\tau} \to X$ is continuous, and let $\iota \colon X \rightarrow \vrm{X}{r}$ be the natural inclusion map.

We will show that $\iota$ and $\pi \circ f$ are homotopy inverses.
Note that $\pi \comp f \comp i = \id_{X}$.
The continuous map $H \colon \cechm{X}{2r} \times I \to \cechm{X}{2r}$ given by $H(x,t) = t\cdot\id_{\cechm{X}{2r}}+(1-t)\iota\comp\pi\comp f$ is well-defined by Lemma~\ref{lem:cech-simplex-ambient} or \ref{lem:cech-simplex-intrinsic} and is the necessary homotopy equivalence from $\iota \circ \pi \circ f$ to $id_{\cechm{X}{2r}}$. Hence $\cechm{X}{2r}$ is homotopy equivalent to $X$.
\end{proof}

In the case of the metric \v{C}ech thickening, the bound $r < 2\tau$ is tight.
For example, consider the zero sphere $\sphere^0 = \{-1,1\} \subseteq \R$.
The reach of $\sphere^0$ is $\tau = 1$.
At scale parameter $r = 2$ we have that $\cechamleq{\sphere^0}{\R}{2} \cong [-1,1]$ is contractible, and hence not homotopy equivalent to $\sphere^0$.

\section{Conclusion}\label{sec:conclusion}
Subsets of Euclidean space of positive reach are a class of objects of particular interest in topological data analysis, and in this paper we have shown that Vietoris--Rips and \v{C}ech thickenings of these spaces recover the same topological information as the space itself.
Moreover, metric Vietoris--Rips and \v{C}ech thickenings retain metric information about the subset, in stark contrast with the classical Vietoris--Rips and \v{C}ech simplicial complexes, which in general are not metrizable.
Furthermore, metric thickenings have the advantage of allowing simpler (and explicit) constructions of the maps realizing homotopy equivalences in analogues of Hausmann's result and the nerve lemma.
Several questions, however, remain open.
In particular, Latschev's theorem~\cite{Latschev} states that if $Y$ is Gromov--Hausdorff close to a manifold $X$, then an appropriate Vietoris--Rips complex of $Y$ is homotopy equivalent to the manifold. A metric analogue for Vietoris--Rips thickenings is currently known only when $Y$ is finite (\cite[Theorem~4.4]{MetricReconstructionViaOptimalTransport}), even though we expect the result to also be true for infinite $Y$.

\bibliographystyle{plain}
\bibliography{MetricThickeningsOfEuclideanSubmanifolds}

\end{document}